\documentclass[11pt]{article}
\usepackage{amsfonts}
\usepackage{mathrsfs,color,amscd,amssymb,enumerate,amsthm,amsmath,bm,graphicx,psfrag,subfigure,unicode-math}
\usepackage{latexsym}
\usepackage{float,fancybox,shapepar,setspace,hyperref}
\usepackage{pgf,tikz}

\makeatletter
\def\leftharpoonfill@{\arrowfill@\leftharpoonup\relbar\relbar}
\def\rightharpoonfill@{\arrowfill@\relbar\relbar\rightharpoonup}
\newcommand\rbjt{\mathpalette{\overarrow@\rightharpoonfill@}}
\newcommand\lbjt{\mathpalette{\overarrow@\leftharpoonfill@}}
\makeatother

\makeatletter

\renewcommand{\@seccntformat}[1]{{\csname the#1\endcsname}{\normalsize .}\hspace{.5em}}
\makeatother

\def \[{\begin{equation}}
\def \]{\end{equation}}

\def \ss {\subseteq}

\def\bqed{ \hfill $\blacksquare$}

\newtheorem{thm}{Theorem}

\newtheorem{claim}{Claim}

\newtheorem{lem}{Lemma}

\newtheorem{conj}{Conjecture}

\usetikzlibrary{arrows}
\begin{document}

\title{A note on long nontrivial cycle in Hamiltonian graphs}
\author{Xiaolin Wang\thanks{ Corresponding author. School of Mathematics and Statistics, Fuzhou University, Fuzhou
350108, P.R. China (Email: xiaolinw@fzu.edu.cn)}~, Jiabao Yang\thanks{School of Mathematics, Nanjing University, Nanjing 210093, P.R. China (Email: jbyang1215@nju.edu.cn)}~, Guangmiao Yu\thanks{School of Mathematics and Statistics, Fuzhou University, Fuzhou
350108, P.R. China (Email: guangmiaoyu@163.com)}~, Ruilin Zheng\thanks{School of Mathematics, Nanjing University, Nanjing 210093, P.R. China (Email: rlzhengmath@163.com)}}
\date{}
\maketitle

\begin{abstract}
Let $G$ be an $n$-vertex graph containing a Hamiltonian cycle and with  minimum degree at least $3$. Gir\~{a}o, Kittipassorn and Narayanan (Israel J. Math., 2019) proved that $G$ contains another cycle of length at least $n-O(n^{4/5})$.  In this paper, we improve their bound to $n-O(n^{2/3})$. Our proof is combined with  a constructive method, which is based on a   poset result, and  a  nonconstructive method. And the bound is best possible under these two methods.

\vskip 2mm

\noindent{\bf Keywords}: Hamiltonian cycle, Minimum degree, Poset
\end{abstract}
{\setcounter{section}{0}

\section{Introduction}\setcounter{equation}{0}

\vskip 2mm
Let $G$ be a simple and connected graph of order $n$.  A \emph{Hamiltonian} cycle is a cycle containing all the vertices in $G$. We say a graph is \emph{Hamiltonian} if it contains a Hamiltonian cycle. Throughout this paper, we suppose that $G$ is Hamiltonian and fix a Hamiltonian cycle $C$ of $G$. A cycle of $G$ is called \emph{nontrivial} if it is distinct from $C$. For an integer $k\geq 1$, we say that $G$ is $k$-regular if every vertex of $G$ has degree $k$. We denote by $\delta(G)$ the minimum degree of $G$.

The problem of finding long cycles distinct from a given Hamiltonian cycle has been studied extensively for more than 80 years; see, for example, \cite{B,G} for surveys. The famous longstanding Sheehan's conjecture \cite{S} states that, for each $k\geq 3$, every $k$-regular Hamiltonian graph has at least two distinct Hamiltonian cycles. An early result in this direction is Smith's theorem, which states that every $3$-regular Hamiltonian graph contains an even number of Hamiltonian cycles through any prescribed edge. Tutte \cite{Tutte} gave the first published proof of this theorem in 1946. Later, in 1978, Thomason \cite{TAG} extended the Smith's theorem to all $k$-regular Hamiltonian graphs with odd $k$ by means of a non-constructive argument, commonly referred to as the ``lollipop'' argument. Twenty years later, Thomassen \cite{TC1,TC2,TC3} proved  Sheehan's conjecture for all $k\geq 300$ by combining the arguments of Thomason \cite{TAG} with the Lov\'asz local lemma \cite{EL}. In 2006, Haxell, Seamone, and Verstra\"ete \cite{HSV} verified Sheehan's conjecture for all $k\geq 23$. Thus, in view of Thomason's result for odd $k$, the only remaining cases are even values $4\leq k\leq 22$. Moreover, Petersen's $2$-factor decomposition theorem \cite{P} shows that it would suffice to establish the conjecture for $k=4$, since this would imply the conjecture for all even $k\geq 4$. For more results related to Sheehan's conjecture, see \cite{GJ,GMZ,JW,K,SW,Z}.

It is natural to ask whether the weaker condition $\delta(G)\geq 3$ already forces a Hamiltonian graph to contain a second Hamiltonian cycle. 
Unfortunately, Entringer and Swart \cite{ES} constructed a Hamiltonian graph in which every vertex has degree $3$ or $4$, but the graph has a unique Hamiltonian cycle. 
Further counterexamples can be found in \cite{F,GJ}. 
It is worth noting that all these counterexamples still contain a long cycle distinct from the Hamiltonian cycle. Recently, Gir\~{a}o, Kittipassorn, and Narayanan \cite{GKN} posed the following conjecture.

\begin{conj}[Gir\~{a}o, Kittipassorn, and Narayanan \cite{GKN}]\label{conj}
Let $G$ be a simple graph with $n$ vertices and $\delta(G) \geq 3$ containing a Hamiltonian cycle. Then there exists a constant $c>0$ such that $G$ contains another cycle of length at least $n-c$.
\end{conj}

Gir\~{a}o, Kittipassorn, and Narayanan \cite{GKN} proved a polynomial-error version of this conjecture: there exists a constant $c>0$ such that $G$ contains another cycle of length at least $n-cn^{4/5}$. 
In this paper, we improve their bound as follows.

\begin{thm}\label{main}
Let $G$ be a simple graph with $n$ vertices and $\delta(G) \geq 3$ containing a Hamiltonian cycle.  Then there exists a constant $c$ such that $G$ contains another cycle of length at least $n-cn^{2/3}$.
\end{thm}

At the end of this section, we introduce another interesting and challenging problem concerning the distribution of cycle lengths in Hamiltonian graphs. 
At the 1999 conference ``Paul Erd\H{o}s and His Mathematics'', Jacobson and Lehel conjectured that, for every $k\geq 3$, every $k$-regular Hamiltonian graph on $n$ vertices contains cycles of at least $\Omega(n)$ distinct lengths; in other words, the number of distinct cycle lengths should be linear in $n$.  The  3-regular graph with $n/6+O(1)$ different cycle lengths can be seen in \cite{BGS}.
In 2016, Verstra\"ete \cite{V2} proposed the following stronger conjecture: every Hamiltonian graph $G$ on $n$ vertices with $\delta(G)\geq 3$ contains cycles of at least $\Omega(n)$ distinct lengths. 
A major result toward these two conjectures was obtained by Buci\'c, Gishboliner, and Sudakov \cite{BGS}, who proved that every Hamiltonian graph $G$ on $n$ vertices with $\delta(G)\geq 3$ contains cycles of at least $n^{1-o(1)}$ distinct lengths. 
They also asked whether, under the same assumptions, $G$ must contain a cycle of length $n-c$ for some constant $c>0$, and suggested that this problem may be useful for approaching the above two conjectures. 
For more results and problems concerning the distribution of cycle lengths in Hamiltonian graphs, see \cite{BGS,MPR}.

\section{Preliminaries}
In this section, we introduce some terminology and helpful theorems that will be used in the proof of Theorem~\ref{main}. Note that $G$ is a simple Hamiltonian graph with a fixed Hamiltonian cycle $C$. 
We call any edge in $E(G)\setminus E(C)$ a \emph{chord}. 
For any simple connected graph $H$, let $X\subseteq V(H)$ and $E_0\subseteq E(H)$. 
We denote by $H[X]$ the subgraph of $H$ induced by $X$, that is, the subgraph with vertex set $X$ and all edges of $H$ whose endpoints both lie in $X$. 
A set $X$ is called an \emph{independent set} in $H$ if no edge of $H$ has both endpoints in $X$. 
A set $X$ is called a \emph{dominating set} in $H$ if every vertex in $V(H)\setminus X$ is adjacent to at least one vertex of $X$.
We also say that $X$ \emph{dominates} $E_0$ in $H$ if every edge in $E_0$ has at least one endpoint in $X$. 
Finally, we denote by $H-E_0$ the subgraph obtained from $H$ by deleting all edges in $E_0$.

The following useful theorem is based on parity-based arguments and has been applied in the study of Sheehan's conjecture; see, for example, \cite{HSV,TC1,TC2,TC3}.

\begin{thm}[Thomassen \cite{TC2}]\label{Thomassen}
Let $H$ be a graph with a Hamiltonian cycle $C$ ($H$ may contain some parallel edges).
If there exists a set $X\subseteq V(H)$ such that $X$ is independent in $C$ and is a dominating set in the subgraph $H-E(C)$, 
then $H$ contains a nontrivial Hamiltonian cycle.
\end{thm}

A \emph{partially ordered set}, or \emph{poset} for short, is a set $S$ together with a partial order, denoted by $``\leq"$.
A \emph{chain} in a poset $S$ is a subset $S_0\subseteq S$ such that any two distinct elements of $S_0$ are comparable.
An \emph{antichain} in a poset $S$ is a subset $S_0\subseteq S$ such that no two distinct elements of $S_0$ are comparable.
A classical theorem in poset theory states that, for any positive integers $a$ and $b$, every poset of size $ab+1$ contains either a chain of size $a+1$ or an antichain of size $b+1$.
For convenience, we present the following revised version of this theorem.

\begin{thm}[Dilworth \cite{D}]\label{poset}
For any positive integers $a$ and $b$, every poset of size $ab$ contains either a chain of size $a$ or an antichain of size $b$.
\end{thm}

It is worth mentioning that Buci\'c, Gishboliner and Sudakov \cite[Lemma~3.1]{BGS} used the classical Erd\H{o}s--Szekeres lemma,  which can be viewed as a special case of Theorem~\ref{poset},  to find two types of chord sets in  Hamiltonian graphs with minimum degree at least 3.

Before defining a poset on the set of all chords of $G$, we introduce some notation. 
The length of a path $P$, denoted by $|P|$, is the number of edges in $P$. 
We choose one  direction of $C$ as the \emph{forward} direction of $C$ and denote it by $\overrightarrow{C}$, and the other direction of $C$ as the \emph{backward} direction of $C$ and denote it by $\overleftarrow{C}$.
We can embed the graph $G$ in the plane such that the Hamiltonian cycle $C$ is the outer bound, and all the chords lie in the inner face of $C$.
Note that every chord $uv$ separates $C$ into two paths, namely $u\overrightarrow{C}v$ and $u\overleftarrow{C}v$. 
We call $V(u\overrightarrow{C}v)$ and $V(u\overleftarrow{C}v)$ the two \emph{domains} of $uv$. These two domains intersect precisely in $\{u,v\}$.

Two vertices are said to be \emph{chord-adjacent} if they are the two endpoints of a chord. And $u$ is called a \emph{chord-neighbor} of $v$ if $uv$ is a chord. 
Two chords $u_1v_1$ and $u_2v_2$ are said to \emph{interlace} if the vertices $u_1,v_1,u_2,v_2$ are distinct and appear in the order $u_1,u_2,v_1,v_2$ or $u_1,v_2,v_1,u_2$ along the forward direction of $C$. 
A set of chords is called \emph{interlacing} if every two chords in the set interlace.

Fix an edge $e\in E(C)$. We define a poset $S_e$ on the set of all chords of $G$ as follows. For a chord $f$, the two domains of $f$ partition the Hamiltonian cycle $C$ into two arcs. Exactly one of these two domains contains both endpoints of $e$; we call this domain the \emph{interior} of $f$ with respect to $e$, and the other domain is called the \emph{exterior} of $f$ with respect to $e$.
For any two distinct  $f_1,f_2\in S_e$, we say $f_1<f_2$ if the interior of $f_1$ is contained in the interior of $f_2$. It is straightforward to verify that $S_e$ is a poset with this partial order.

The following theorem shows that a large antichain in the poset of chords of $G$ guarantees the existence of a long nontrivial cycle in $G$.

\begin{thm}\label{poset-based}
Let $G$ be a simple $n$-vertex graph consisting of a Hamiltonian cycle $C$ together with $k$ chords. 
Suppose that there exists an edge $e\in E(C)$ and a constant $c_0>0$ such that the maximum size of a chain in $S_e$ is less than $\lfloor c_0\sqrt{k}\rfloor$. 
Then $G$ contains a nontrivial cycle of length at least
$n-\frac{8c_0n}{\sqrt{k}}.$
\end{thm}

\noindent{\bf Proof.}
Since $\lfloor c_0\sqrt{k}\rfloor \cdot\lfloor\sqrt{k}/c_0\rfloor \le k$, we may apply Theorem~\ref{poset} to  $S_e$ of size $k$. By assumption, $S_e$ contains no chain of size $\lfloor c_0\sqrt{k}\rfloor$.  Therefore, Theorem~\ref{poset} implies that $S_e$ contains an antichain $A$ of size at least $\lfloor\sqrt{k}/c_0\rfloor \geq \sqrt{k}/2c_0$.
We first describe the structure of $A$, as illustrated in Figure~\ref{A}. For convenience, suppose $e=xy$ and $y , x$ appear in the forward direction of $C$.

\begin{figure}[!htb]
\centering
\includegraphics[height=0.4\textwidth]{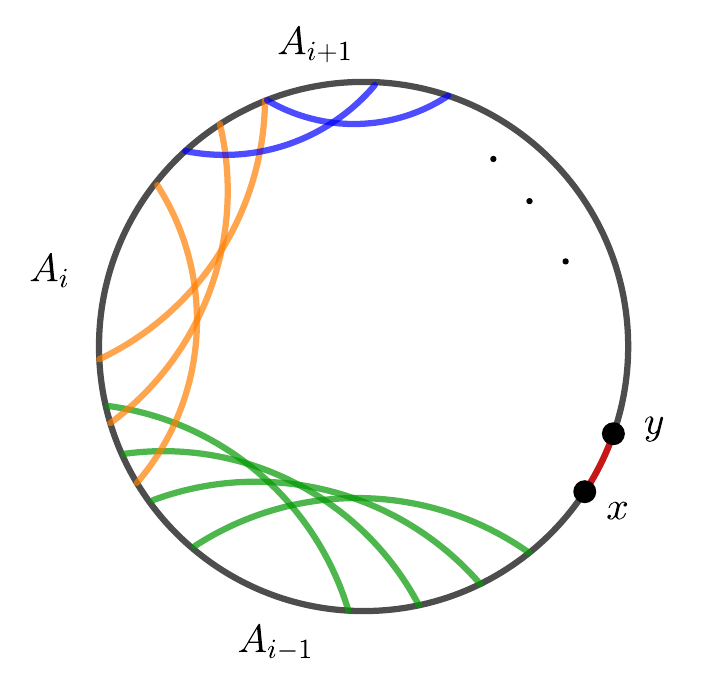}
\caption{The structure of the antichain $A$.}
\label{A}
\end{figure}

Starting from $x$, we construct a collection of pairwise disjoint interlacing chord subsets of $A$ by scanning along the forward direction of $C$. 
Suppose that disjoint interlacing chord subsets $A_1,\ldots,A_i$ of $A$ have already been constructed. 
If $A\setminus \left(\cup_{j=1}^{i} A_j\right)=\emptyset$, then the process terminates. Otherwise, choose a chord $uv$ in $A\setminus \left(\cup_{j=1}^{i} A_j\right)$ so that $u$ is the first endpoint encountered when we search the vertices starting from $x$ along the forward direction of $C$; here $u$ may coincide with $x$. 
If no other chord in $A\setminus \left(\cup_{j=1}^{i} A_j\right)$ interlaces $uv$, 
then let $A_{i+1}=\{uv\}$. Otherwise, let $A_{i+1}$ be the maximal interlacing chord set containing $uv$.
Since $|A|\leq k$, this process terminates after finitely many steps. Suppose that it terminates after $s$ steps. Then $A_1,\ldots,A_s$ form a partition of $A$.

For each $1\leq i\leq s$, we write $A_i=\{u_1^iv_1^i,\ldots,u_{a_i}^iv_{a_i}^i\}$, where $a_i\geq 1$ and the sequence $u_1^i,\ldots,u_{a_i}^i,v_1^i,\ldots,v_{a_i}^i$ appear in this order along the forward direction of $C$. 
Let $U_i=\{u_1^i,\ldots,u_{a_i}^i\}$ and $V_i=\{v_1^i,\ldots,v_{a_1}^i\}$.

\begin{claim}\label{1}
Suppose $s\geq 2$ and $i<s$. Then $U_{i+1}\ss V(v_1^{i}\overrightarrow{C}y)$.
\end{claim}

\noindent{\bf{Proof.}}
Let $u\in U_{i+1}$, and let $uv$ be the chord in $A_{i+1}$ incident with $u$. Then $v\in V_{i+1}$.  By the choice of $A_i$, we have $u\in V(u_1^{i}\overrightarrow{C}y)$. Suppose to the contrary that $u\in V(u_1^{i}\overrightarrow{C}v_1^i)-v_1^i$.
If $a_i=1$  and $u=u_1^i$, no matter what $v$ lies in, $uv$ and $u_1^iv_1^i$ are comparable, a contradiction to the fact that $A$ is an antichain.
If $u\in V(u_j^i\overrightarrow{C}u_{j+1}^i)$ for some $1\leq j<a_i$, by the definition of the antichain $A$, then $v\notin V(u\overrightarrow{C}v_j^i)\cup V(v_{j+1}^i\overrightarrow{C}y)$. It follows that $v$ must lie in the internal of $v_j^i\overrightarrow{C}v_{j+1}^i$. Also by the definition of the antichain $A$, $u$ must lie in the internal of $u_j^i\overrightarrow{C}u_{j+1}^i$. Consequently, $A_i\cup \{uv\}$ is an interlacing chord set strictly larger than $A_i$, contradicting the choice of $A_i$.
It remains to consider that $u$ is an internal vertex in $u_{a_i}^i\overrightarrow{C}v_1^i$. By the definition of the antichain $A$, $v\notin V(u\overrightarrow{C}v_{a_i}^i)$. It implies that $v\in V(v_{a_i}^i\overrightarrow{C} y) - v_{a_i}^i$. However, $A_i\cup \{uv\}$ is a larger interlacing chord set, contradicting the choice of $A_i$.
Therefore, $U_{i+1}\ss V(v_1^{i}\overrightarrow{C}y)$, as required.
\bqed
\vskip 2mm

\begin{claim}\label{2}
Suppose $s\geq 2$ and $i<s$. Then $V_{i+1}\ss V(v_{a_i}^{i}\overrightarrow{C}y)-v_{a_i}^{i}$.
\end{claim}

\noindent{\bf{Proof.}}
Let $v\in V_{i+1}$, and let $uv$ be the chord in $A_{i+1}$ incident with $v$, where $u\in U_{i+1}$. By Claim \ref{1}, $u\in V(v_1^{i}\overrightarrow{C}y)$.
Then $v\in V(v_1^i\overrightarrow{C}y)$.
 It is easy to see that $v\notin V(v_1^i\overrightarrow{C}v_{a_i}^i)$. Otherwise, the chords $uv$ and $u_{a_i}^iv_{a_i}^i$ would be comparable. which is a contradiction to the fact that $A$ is an antichain. Thus, $v\in V(v_{a_i}^{i}\overrightarrow{C}y)-v_{a_i}^{i}$ implies $V_{i+1}\ss V(v_{a_i}^{i}\overrightarrow{C}y)-v_{a_i}^{i}$.
This proves the claim.
\bqed
\vskip 2mm

Let $P_i=u_1^i\overrightarrow{C}v_{a_i}^i$.

\begin{claim}\label{3}
Suppose $s\geq 3$ and $3\leq i+2\leq s$. Then $P_{i}$ and $P_{i+2}$ are vertex-disjoint.
\end{claim}

\noindent{\bf{Proof.}}
By Claim \ref{2}, $v_1^{i+1}\in V(v_{a_i}^i\overrightarrow{C}y)-v_{a_i}^i$. By Claim \ref{1}, $u_1^{i+2}\in V(v_1^{i+1}\overrightarrow{C}y)$. 
Then $u_1^{i+2}\in V(v_{a_i}^i\overrightarrow{C}y)-v_{a_i}^i$. 
Hence,  $P_{i}$ and $P_{i+2}$ are vertex-disjoint. \bqed

\vskip 2mm

For distinct indices $i$ and $j$, we say that $A_i$ and $A_j$ are \emph{non-intersecting} if $P_i$ and $P_j$ are vertex-disjoint.
By Claim \ref{3}, 
 the members of each of the two families $\{A_1,A_3,\ldots\}$ and $\{A_2,A_4,\ldots\}$ are pairwise non-intersecting.
Since $|A|\geq \sqrt{k}/2c_0$, one of these two families has total size at least $\sqrt{k}/4c_0$. 
We choose such a subfamily and, after relabelling its members, denote it by $A_1,\ldots,A_t$, where $t\geq 1$.  Then the sets $A_1,\ldots,A_t$ are pairwise non-intersecting and satisfy $\sum_{i=1}^{t}|A_i|\geq \sqrt{k}/4c_0$.
Recall that $|A_i|=a_i\geq 1$, we have $\sum_{i=1}^{t}a_i\geq \sqrt{k}/4c_0$.

If $a_i=1$, let $L_1^i=|u_1^i\overrightarrow{C}v_1^i|$.
If $a_i\geq 2$, then for every $1\leq j<a_i$, let $L_j^i=|u_j^i\overrightarrow{C}u_{j+1}^i|+|v_j^i\overrightarrow{C}v_{j+1}^i|$.
Let $\ell$ denote the total number of all $L_j^i$ defined in this way. For each $i$, the number of such $L_j^i$ is $1$ if $a_i=1$, and is $a_i-1$ if $a_i\geq 2$. In either case, it is at least $a_i/2$. Therefore, $\ell\geq( \sum_{i=1}^t a_i)/2\geq \sqrt{k}/8c_0$.
Since the sets $A_1,\ldots,A_t$ are pairwise non-intersecting, the arcs of $C$ appearing in the definitions of the $L_j^i$ are pairwise edge-disjoint. 
Hence the sum of all these $L_j^i$ is at most $n$. 
It follows that there exists some pair $(i,j)$ such that $L_j^i\leq n/\ell\leq 8c_0n/\sqrt{k}.$
If $a_i=1$, then we find a nontrivial cycle $u_1^iv_1^i\overrightarrow{C}u_1^i$ of length at least $n+1-8c_0n/\sqrt{k}$. If $a_i\geq 2$, then we find a nontrivial cycle $u_j^iv_j^i\overleftarrow{C}u_{j+1}^iv_{j+1}^i\overrightarrow{C}u_j^i$ of length at least $n+2-8c_0n/\sqrt{k}$.
In both cases, $G$ contains a nontrivial cycle of length at least
$$n-\frac{8c_0n}{\sqrt{k}}.$$
This completes the proof.
 \bqed

\section{Proof of Theorem \ref{main}}

For convenience, we may assume that $G$ is edge-minimal with respect to the assumptions of Theorem~\ref{main}; that is, no two vertices of degree at least $4$ are chord-adjacent. 
Indeed, if a chord joins two vertices both of degree at least $4$, then deleting this chord preserves the Hamiltonian cycle $C$ and the condition $\delta(G)\geq 3$.
An \emph{interlacing pair} is an interlacing chord set of size two. 
We say that a set of chords is \emph{independent} if no two chords in the set share a common endpoint. 
Let $I$ be a maximum independent chord set in $G$ such that  it can be partitioned into  interlacing pairs. Suppose $|I|=2m$. 
Let $G_0$ be the graph consisting of the Hamiltonian cycle $C$ together with all chords in $I$.
We will prove Theorem~\ref{main} by considering cases according to the size of $I$.
\vskip 2mm

\noindent{\bf Case 1.} $m\geq n^{2/3}$.
Let $G'\ss G_0$ be a graph consisting of $C$ and any chord subset $S\ss I$ of size at least $m/8$.
For any fixed edge $e\in E(C)$, the set $S$ inherits the poset structure defined above from the poset $S_e$ of chords in $G'$.
Suppose first that, for some constant $c_1>0$, the maximum size of a chain in this poset $S$ is at most $c_1n^{1/3}$. Since $m\geq n^{2/3}$, we have $c_1n^{1/3}\leq c_1\sqrt{m}\leq \sqrt{8}c_1\sqrt{m/8}\leq \lfloor4c_1\sqrt{|S|}\rfloor$.
Applying Theorem \ref{poset-based} on $(G,k,c_0)=(G',|S|,4c_1)$, then we can find a nontrivial cycle of length at least 
$$ n-\frac{32c_1n}{\sqrt{|S|}}\geq n-\frac{32c_1n}{\sqrt{m/8}}\geq n-cn^{2/3}$$
for some constant $c>0$, where the last inequality follows from $m\geq n^{2/3}$.
Since $G'\subseteq G$, this cycle is also a nontrivial cycle in $G$.

Thus, in the remainder of this case, we may consider that 
\vskip 1mm

\noindent{\bf{(I)}} for every edge \(e\in E(C)\) and every \(S\subseteq I\) with
\(|S|\ge m/8\), the poset induced by \(S\) in \(S_e\) contains a chain
of size at least \(c_1n^{1/3}\)  for some constant $c_1$.

Recall that $I$ is an independent chord  set, and hence no two chords in $I$ share an endpoint. 
We say a triple $\{u_1v_1,u_2v_2,u_3v_3\}$ in $I$ is \emph{tight} if 
$u_1,u_2,u_3,v_3,v_2,v_1$ appear in this order along the forward direction of $C$, and
$$|u_1\overrightarrow{C}u_3|+{|v_3\overrightarrow{C}v_1|}\leq \frac{3n^{2/3}}{c_1}.$$
We now present two claims under the assumption \textup{(I)}.

\begin{claim}[Gir\~{a}o, Kittipassorn, and Narayanan \cite{GKN}, Claim 3.3]\label{claim3.3}
If $G$ contains two tight triples whose middle chords interlace, then
$G$ contains a nontrivial cycle of length at least $n-c_2n^{2/3}$ for some constant $c_2$.
\end{claim}

\begin{claim}[Gir\~{a}o, Kittipassorn, and Narayanan \cite{GKN}, Claim 3.4]\label{claim3.4}
For any $K \geq 1/2$ and any $S_0\ss I$ with $|S_0|=Km$,   $S_0$ contains $Km/4$ pairwise
disjoint tight triples.
\end{claim}

Although our assumption (I) and the definition of a tight triple is not the same as  \cite{GKN} (Page 275),     the proofs of the corresponding claims carry over verbatim after replacing $m^{1/3}$ in \cite{GKN} by $8c_1n^{1/3}$. 
We therefore omit the tedious details.
By Claim~\ref{claim3.3}, together with Claim~\ref{claim3.4} applied with $(K,S_0)=(2,I)$, we may assume that $I$ contains $m/2$ pairwise disjoint tight triples, denote by $I_1$, and all the middle chords of $I_1$ are independent and not pairwise interlacing. 
Let $I_2$ be the set of all middle chords of the triples in $I_1$. 
Since $I$ consists of $m$ interlacing pairs, each chord in $I_2$ has a distinct chord in $I$ that interlaces it. Let $I_3$ denote the set of these interlacing chords. Then $|I_3|=m/2$ and $I_2\cap I_3=\emptyset$.
Applying Claim~\ref{claim3.4} again with $(K,S_0)=(1/2,I_3)$, we obtain $m/8$ new pairwise disjoint tight triples. Let $T_3$ be one of these new tight triples, and let $e_3$ be its middle chord. Since $e_3\in I_3$, there exists a chord $e_2\in I_2$ that interlaces $e_3$. Let $T_2$ be the tight triple in $I_1$ whose middle chord is $e_2$. Then $T_2$ and $T_3$ satisfy the condition in Claim~\ref{claim3.3}, and we are done.
\vskip 2mm

\noindent{\bf{Case 2.}}
$m\leq n^{2/3}$.

Let $X_I$ be the set of vertices incident with  chords in $I$. 
Then $|X_I|=4m$.
Deleting from $G$ all chords incident with $X_I$, we obtain a graph $G_1$. Then each vertex of degree 2 in $G_1$ has a chord-neighbor in $X_I$.
By the maximality of $I$, it is straightforward to check that $G_1$ is a $2$-connected outerplanar graph.
We call a chord $xy$ of $G_1$ \emph{minimal} if one of its two domains contains no pair of endpoints of any other chord in $G_1$. And we denote this domain as \emph{minimal} domain of $xy$.
Let $M$ be the number of minimal chords in $G_1$. Since $G_1$ is a $2$-connected outerplanar graph,  $M\geq 2$ when $G_1$ contains at least two chords, and
  all  minimal domains are  pairwise disjoint, except the endpoints of  minimal chords.
If $M\geq \sqrt{n}$, then there exists a minimal chord $xy$ such that one of the domains determined by $xy$ has size at most $n/\sqrt{n}=\sqrt{n}$. Hence there exists a constant $c_3$ such that either $xy\overrightarrow{C}x$ or $xy\overleftarrow{C}x$ has length at least $n-c_3\sqrt{n}$. This gives a nontrivial cycle in $G$, completing the proof. Therefore we may suppose that $M\leq \sqrt{n}$.

Let $v\in V(G_1)$ such that $v$ is chord-adjacent to two consecutive vertices $v_x,v_y$ of $C$  in $G_1$.
Since $d_G(v)\geq d_{G_1}(v)\geq 4$, by the minimality of $G$, the degrees of both $v_x$ and $v_y$ in $G$ and in $G_1$ are equal 3. In particular, $v_x$ and $v_y$ have no chord-neighbor in $X_I$, and both $v_x$ and $v_y$ cannot be chord-adjacent to two consecutive vertices of $C$.
We obtain $G_2$ from $G_1$ by repeatedly applying the following operation: if there exists a vertex $v$ of $G_1$ that is chord-adjacent to two consecutive vertices $v_x,v_y$ of $C$, then contract $v_xv_y$ to a vertex $v'$, while two edges $vv_x,vv_y$ become one edge  $vv'$. Then $G_2$ is also simple and Hamiltonian. Denote by $C_2$ the Hamiltonian cycle obtained from $C$ after these contractions.
It is easy to see that $G_2$ is still a $2$-connected outerplanar graph, $G_2$ has at most $M$ minimal chords,  every contracted vertex has degree $3$ in $G_2$, and each vertex of degree $2$ in $G_2$ has a neighbor in $X_I$. In particular, no contracted  vertex has a chord-neighbor in $X_I$, and no two contracted vertices are chord-adjacent.

Observe that each vertex of $G_2$ is either a vertex of $G$ or a contracted vertex corresponding to more than one vertex of $G$. In particular,  no contracted vertices belong to $X_I$. We call a contracted vertex \emph{red} if it is near to a vertex of $X_I$ in $C_2$, and \emph{blue} otherwise. Let $r$ denote the number of red vertices in $G_2$.  Then $r\leq 2|X_I|=8m$. Although our definitions of red and blue vertices are different from Gir\~{a}o, Kittipassorn, and Narayanan \cite{GKN}, we can also use the  
 following lemma proved by them.

\begin{lem}[Gir\~{a}o, Kittipassorn, and Narayanan \cite{GKN}, Lemma 3.5]\label{lem3.5}
Let $G$ be a graph  with a  Hamiltonian cycle $C$ and 
with the property that no two chords of $G$ interlace. Suppose that no vertex of $G$
is chord-adjacent to two consecutive vertices of $C$, and that no two vertices of $G$
of degree greater than $3$ are chord-adjacent. Also, assume that there are disjoint subsets
$R$ and $B$ of $V(G)$ (whose elements we shall call red and blue respectively) such
that

(I) every vertex in $R\cup B$ has degree $3$;

(II) no two vertices in $R\cup B$ are chord-adjacent.

Then, writing $M \geq 2$ for the number of minimal chords in $G$ and setting $r = |R|$,
there exists a set $S \ss V(G)$  of vertices such that

(1) $S$ dominates the chords of $G$;

(2) $S$ contains no red vertices;

(3) $S$ contains at most $r+M-2$ pairs of consecutive vertices of $C$, and
none of these pairs contains a blue vertex.
\end{lem}

If $G_2$ has no minimal chord, then we can deduce that $G_2=G_1=C$ and let $S=\emptyset$. Hence, $S$  satisfies (1)-(3) in Lemma \ref{lem3.5} when we take $G=G_2$ in Lemma \ref{lem3.5}. If $G_2$ has exactly one minimal chord, then $G_2$ is a union of $C_2$ and a chord $uv$, and let $S$ be the set of the vertex $u$ or $v$, which is not a red vertex. It is easy to check that $S$ also satisfies (1)-(3) in Lemma \ref{lem3.5} when we take $G=G_2$ in Lemma \ref{lem3.5}.
If $G_2$ has at least 2 minimal chords, then  $G_2$ with red vertices as $R$ and blue vertices as $B$ satisfies Lemma \ref{lem3.5}. Hence there exists a set $S\ss V(G_2)$ satisfying the properties stated in Lemma \ref{lem3.5}.

Note that no red or blue vertex has a chord-neighbor in $X_I$. We obtain $G_3$ by adding back the chords incident with vertices in $X_I$. 
Since every vertex of degree $2$ in $G_2$ has a chord-neighbor in $X_I$, by Lemma \ref{lem3.5}(1), the set $X=X_I\cup S$ is a dominating set of $G_3-E(C_2)$ and contains no red vertices by Lemma \ref{lem3.5}(2). 
Now we investigate the  pairs of consecutive vertices of $C_2$ containing in $X$.
Let $Q_1$ be the set of pairs of consecutive vertices of $C_2$ containing in $X$, each of which has at least one vertex in $X_I$. Since $|X_I|=4m$, $|Q_1|\leq 8m$. By the definition of  red vertices and Lemma \ref{lem3.5}(2), every vertex of each pair in $Q_1$ is a vertex in $V(G)$. Let $Q_2$ be the set of pairs of consecutive vertices  of $C_2$, each of which have both vertices in $S$. By Lemma \ref{lem3.5}(2)(3), every vertex in each pair of $Q_2$ is a vertex in $V(G)$ and  $|Q_2|\leq r+M-2\leq 8m+M-2$. Then the number of pairs of consecutive vertices of $C_2$ in $X$ is at most $|Q_1|+|Q_2|\leq 16m+M-2$, and every vertex of each pair of consecutive vertices of $C_2$ in $X$ is a vertex in $V(G)$.

In order to use Theorem \ref{Thomassen}, we need to ensure that $X$ is independent in the subgraph spanned by the edges of $C_2$. To this end, we repeatedly contract an edge $x_1x_2\in E(C_2)$ with $x_1,x_2\in X$ to a new vertex $x'$. So, if $x_1$ and $x_2$ have a common neighbor $x$, then the contraction may create two parallel edges between $x$ and the contracted vertex.
Denote by $G_4$ the resulting  graph (maybe have some parallel edges) and by $C_4$ its Hamiltonian cycle corresponding to $C_2$. 
Then $X$ is transformed into a set $X_4$ in $G_4$.
It is easy to see that $G_4$ and $X_4$ satisfy the conditions of Theorem \ref{Thomassen}. Hence there exists a nontrivial Hamiltonian cycle $C_4'$ containing at least one edge not in $C_4$. 

Lifting $C_4'$ back to $G_3$, we obtain a nontrivial cycle $C_3'$ in $G_3$ that misses at most  $16m+M-2$ vertices that are also vertices of $G$.
Recall that $G_1$ is obtained from $G$ by deleting some chords incident with $X_I$; $G_2$ is obtained by contracting certain vertex sets, none of which has a chord-neighbor in $X_I$, into red or blue vertices; and $G_3$ is obtained by adding back the chords incident with $X_I$. Thus, by expanding all red and blue vertices in $G_3$, we recover $G$.
We now describe how to obtain a large nontrivial cycle $C'\ss G$ from $C_3'$. Consider any path $xvy\ss C_3'$ where $v$ is a red or blue vertex of $G_3$. Let $v'$ be the unique chord-neighbor of $v$ in $G_3$, and let $v_1,\ldots,v_t$ be the corresponding consecutive vertices of $C$ represented by $v$. Note that $d_{G_3}(v)=3$. If $xvy\ss C_3$, then we replace $xvy$ with the path $xv_1\cdots v_ty$. Otherwise, $x=v'$ or $y=v'$. By symmetry, suppose that $x=v'$. Then we replace $v'vy$ with the path $v'v_1\cdots v_ty$.

By the above construction of $C'$, all the vertices corresponding to red or blue vertices in $C_3'$ will be contained in $C'$. That is, no more vertices in $G$ will be missed when we recover $C_3'$ to $C'$. 
 We can therefore find a nontrivial cycle in $G$ missing at most 
$$16m+M-2\leq c_4n^{2/3}$$
vertices for some constant $c_4$. This completes the proof. \bqed

\vskip 5mm
\noindent{\bf\large Remark.}
Since we use a constructive method based on a poset result in Case 1 and a nonconstructive method in Case 2, the bound is best possible by using our methods. hence,  new methods should be used  to approach Conjecture \ref{conj}.

\vskip 5mm
\noindent{\bf\large Acknowledgements}
\vskip 2mm

  Wang is supported by 
 National Key R\&D Program of China under grant number 2023YFA1010202 and NSFC under grant number 12401447.

\end{document}